\documentclass[12pt,reqno]{amsart}

\usepackage[T1]{fontenc}
\usepackage[utf8]{inputenc}
\usepackage[english]{babel}
\usepackage{amsthm,amssymb,amsmath}
\usepackage{hyperref}

\newcommand{\Syl}{\mathrm{Syl}}

\newtheorem{Theorem}{Theorem}[section]
\newtheorem{Lemma}[Theorem]{Lemma}
\newtheorem{Corollary}[Theorem]{Corollary}
\newtheorem{Proposition}[Theorem]{Proposition}

\newtheorem{athm}{Theorem}

\theoremstyle{Definition}

\newtheorem{Example}{Example}
\newtheorem{Conjecture}{Conjecture}

\theoremstyle{Remark}
\newtheorem{Remark}{Remark}

\DeclareMathOperator{\Aut}{Aut}

\title{Simple Skew Braces with Cyclic Sylow Subgroups}

\author[Marco Damele]{Marco Damele}
\thanks{Dipartimento di Matematica, Universit\`a di Cagliari, Via Ospedale 72, 09124 Cagliari, Italy; \texttt{marco.damele@unica.it}; ORCID 0009-0008-3088-5766}

\date{}

\usepackage{amsthm}

\begin{document}

\maketitle
\vspace{-1.8em}

\begin{abstract}
We study simplicity and splitting phenomena in finite skew braces
under cyclicity assumptions on Sylow subgroups. We first classify finite simple skew braces whose
multiplicative group is a \(Z\)-group. If the additive group is
soluble, then the skew brace is either trivial of prime order or
isomorphic to one of the two simple skew braces of order \(12\) with
additive group \(A_4\) and multiplicative group \(C_3\rtimes C_4\). If
the additive group is insoluble, then it is necessarily isomorphic to
\(\operatorname{PSL}_2(p)\) for some prime \(p\geq5\). This conclusion
is sharp, since such examples exist for every prime \(p\geq5\).

We then consider the more general situation in which only a Sylow
subgroup corresponding to the smallest prime divisor \(p\) of the
order is assumed to be cyclic. Under suitable hypotheses on the
additive or multiplicative Sylow \(p\)-subgroup, we prove that the skew
brace contains a Hall \(p'\)-ideal and splits as a semidirect product
of this ideal with a Sylow \(p\)-subbrace. As a consequence, every
finite simple skew brace satisfying one of these hypotheses is trivial
of prime order. Moreover, as a consequence of our splitting theorem, we verify Byott's solvability conjecture for finite skew braces whose additive group has a cyclic Sylow $2$-subgroup.

\end{abstract}

\medskip


\medskip


\section{Introduction}
A \emph{skew brace} is a triple \(B=(B,+,\cdot)\) such that
\((B,+)\) and \((B,\cdot)\) are groups on the same underlying set \(B\),
called respectively the additive group and the multiplicative group of
\(B\), and
\[
a\cdot(b+c)=a\cdot b-a+a\cdot c
\]
for all \(a,b,c\in B\). 

Skew braces were introduced by Guarnieri and Vendramin in
\cite{Guarnieri_2016} as a non-commutative generalization of the
braces introduced by Rump in \cite{Rump2007}.
They provide an algebraic framework for studying non-degenerate
bijective set-theoretic solutions of the Yang--Baxter equation.
Since their introduction, skew braces have also appeared naturally in
connection with regular subgroups of holomorphs, bijective
\(1\)-cocycles, exact factorizations of groups, Hopf--Galois
structures and matched products; see, for instance,
\cite{Guarnieri_2016,SmoktunowiczVendramin2018}.

A natural problem in the structure theory of skew braces is
to understand how properties of either of the two underlying groups
influence the internal structure of the skew brace.

To the best of the author's knowledge, only a few results are currently available in this direction. Structural restrictions for two-sided skew braces were obtained in \cite{TrappeniersTwoSided,DamelePerfectTwoSided}. In \cite{DameleZGroups}, the author initiated the systematic study of finite skew braces whose additive group is a \(Z\)-group and proved, in particular, that they are supersoluble.

Among the structural properties considered in this context,
simplicity occupies a particularly important position. Recall that a
skew brace is called \emph{simple} if it is non-zero and has no
non-zero proper ideals. The construction and classification of finite
simple braces have been investigated by several authors, using
extensions, matched products and asymmetric products; see
\cite{BachillerSimpleBraces,
BachillerCedoJespersOkninski,
CedoJespersOkninski}.
More recently, Byott constructed an infinite family of finite simple
skew braces which are not braces and do not arise directly from
non-abelian simple groups; see \cite{ByottSimpleSkewBraces}.

Despite these advances, a general classification of finite simple skew
braces appears to be far beyond reach. It is therefore natural to ask
which group-theoretical properties of \((B,+)\) or \((B,\cdot)\)
force the existence of a non-trivial proper ideal.

A first basic example is provided by skew braces whose multiplicative
group is abelian. In this case the skew brace is two-sided. The
structure theorem for finite simple two-sided skew braces obtained by
Trappeniers in \cite{TrappeniersTwoSided} then implies that a finite
simple skew brace with abelian multiplicative group is necessarily
isomorphic to \(C_p\), endowed with the trivial skew brace structure,
for some prime \(p\).

In \cite{DameleZGroups}, the author initiated the systematic study of
finite skew braces for which one of the two underlying groups is a
\(Z\)-group. It was proved there that every finite skew brace whose
additive group is a \(Z\)-group is supersoluble. In particular, such a
skew brace can be simple only in the trivial prime-order case.

Moreover, it was shown that, for skew braces of odd order, the
additive group is a \(Z\)-group if and only if the multiplicative group
is a \(Z\)-group. Consequently, finite skew braces of odd order whose
multiplicative group is a \(Z\)-group are supersoluble as well.

The purpose of the present paper is to study the remaining, and
substantially more delicate, even-order case. More precisely, we
investigate finite simple skew braces whose multiplicative group is a
\(Z\)-group.

Our first main result gives a complete classification when the
additive group is soluble. We denote by \(S_{12,22}\) and
\(S_{12,23}\) the two simple skew braces of order \(12\) appearing in
the classification of skew braces of small order in
\cite{KonovalovSmoktunowiczVendramin}. Both have additive group
isomorphic to \(A_4\) and multiplicative group isomorphic to the
non-abelian semidirect product \(C_3\rtimes C_4\).

\begin{athm}\label{thm:A}
Let \(B\) be a finite simple skew brace of soluble type such
that \((B,\cdot)\) is a \(Z\)-group. Then $B \simeq C_{p}$ or $B\simeq S_{12,22}$ or $B\simeq S_{12,23}$.
\end{athm}

The proof combines restrictions on the additive Sylow subgroups with
the structure of finite supersoluble groups. We first show that the
additive group has no non-trivial normal subgroup of odd order. Its
Fitting subgroup is therefore a \(2\)-group and contains a small
characteristic elementary abelian subgroup.

On the multiplicative side, a \(Z\)-group of even order has a normal
Hall \(2'\)-subgroup and a cyclic Sylow \(2\)-subgroup. Comparing the
small characteristic \(2\)-subgroups of the additive group with the
central \(2\)-subgroups of the multiplicative group produces a proper
ideal, except in one configuration. The remaining case forces
\(|B|=12\), and the conclusion follows from the classification of
simple skew braces of that order.

We next consider the case in which the additive group is insoluble.
The classification of Suzuki and Wong of finite insoluble groups whose
Sylow subgroups of odd order are cyclic and whose Sylow
\(2\)-subgroups contain a cyclic subgroup of index at most \(2\) plays
a decisive role.

\begin{athm}\label{thm:B}
Let \(B\) be a finite simple skew brace such that \((B,\cdot)\) is a
\(Z\)-group. If \((B,+)\) is insoluble, then
$(B,+)\simeq\operatorname{PSL}_2(p)$
for some prime \(p\geq5\).
\end{athm}

Indeed, the Suzuki--Wong theorem provides a normal subgroup of index
at most \(2\) whose non-soluble component is isomorphic either to
\(\operatorname{PSL}_2(p)\) or to \(\operatorname{SL}_2(p)\),
together with a coprime soluble \(Z\)-group factor. We show that the
latter factor is characteristic and eliminate it by means of the
generalized lambda map.

A characteristic subgroup of index \(2\) would also give rise to a
proper ideal. This reduces the additive group to
\(\operatorname{PSL}_2(p)\) or \(\operatorname{SL}_2(p)\). Finally,
the central involution of \(\operatorname{SL}_2(p)\) is forced into the
socle by comparing the cyclic Sylow \(2\)-subgroup of
\((B,\cdot)\) with the \(2\)-subgroups of
\(\operatorname{Aut}(\operatorname{SL}_2(p))\). This excludes the
quasisimple case.

The conclusion of Theorem~\ref{thm:B} is sharp. Using a construction of Tsang
\cite{TsangSimpleType}, one obtains, for every prime \(p\geq5\), a
finite simple skew brace \(B_p\) such that
$(B_p,+)\simeq\operatorname{PSL}_2(p)$
and \((B_p,\cdot)\) is a \(Z\)-group.

Theorems~\ref{thm:A} and~\ref{thm:B} show that the assumption that the entire
multiplicative group is a \(Z\)-group is remarkably restrictive. In
the second part of the paper, we weaken this assumption and require
only that a Sylow subgroup corresponding to the smallest prime divisor
of \(|B|\) be cyclic.

Our third main result is the following unified splitting theorem.

\begin{athm}\label{thm:C}
Let \(B\) be a finite skew brace, and let \(p\) be the smallest prime
divisor of \(|B|\). Suppose that one of the following conditions
holds:
\begin{enumerate}
    \item a Sylow \(p\)-subgroup of \((B,+)\) is cyclic;
    \item \(p\) is odd and a Sylow \(p\)-subgroup of
    \((B,\cdot)\) is cyclic;
    \item \(p=2\), \(3\nmid |B|\), and a Sylow \(2\)-subgroup of
    \((B,\cdot)\) is cyclic.
\end{enumerate}
Then \(B\) contains a Hall \(p'\)-ideal \(H\). Moreover, there exists
a Sylow \(p\)-subbrace \(P\) of \(B\) such that
$B\simeq H\rtimes P.$
\end{athm}

The proof first produces a normal Hall \(p'\)-subgroup of the additive
group. The generalized lambda map is then used to prove that this
subgroup is an ideal. Once the Hall ideal has been detected, the
Schur--Zassenhaus theorem for finite skew braces
\cite{DameleSchurZassenhaus2026,
FerraraTrombettiSchurZassenhaus2026}
provides a Sylow \(p\)-subbrace complement.

For odd \(p\), the passage from a cyclic multiplicative Sylow
\(p\)-subgroup to a cyclic additive Sylow \(p\)-subgroup follows from
Kohl's result on elements of maximal order in holomorphs. When \(p=2\),
the corresponding theorem of Byott allows the additive Sylow
\(2\)-subgroup to be cyclic, dihedral, generalized quaternion or
isomorphic to \(C_2\times C_2\).

The additional hypothesis \(3\nmid |B|\) in the third case is
essential. The simple skew braces \(S_{12,22}\) and \(S_{12,23}\) have
cyclic multiplicative Sylow \(2\)-subgroups, but they do not admit a
Hall \(2'\)-ideal.

As an immediate consequence of Theorem~\ref{thm:C}, a finite simple skew brace
satisfying any of its hypotheses is necessarily isomorphic to \(C_p\).
Thus, the classification results of the first part and the splitting
criteria of the second part describe two complementary aspects of the
same phenomenon: cyclic Sylow subgroups frequently force the existence
of proper ideals, while the exceptional configurations are governed by
a very restricted collection of finite groups. As a further consequence of Theorem~\ref{thm:C}, we obtain a new
positive instance of Byott's solvability conjecture, which predicts
that the multiplicative group of a finite skew brace is soluble
whenever its additive group is soluble. More precisely, we prove that
the conjecture holds whenever the additive group has a cyclic Sylow
$2$-subgroup.

The paper is organized as follows. In Section~\ref{Preliminaries}, we collect the
necessary preliminaries on skew braces, generalized lambda maps,
\(Z\)-groups, regular subgroups of holomorphs and supersoluble groups.
In Section~\ref{Section simple multiplicative Z group}, we study finite simple skew braces whose multiplicative
group is a \(Z\)-group. The soluble and insoluble additive cases are
treated separately, yielding Theorems~\ref{thm:A} and~\ref{thm:B}. 
In Section~\ref{Section Splitting criteria}, we prove the unified
splitting criterion of Theorem~\ref{thm:C}, derive its consequences
for simple skew braces, establish a new positive instance of Byott's
solvability conjecture, and discuss the necessity of the additional
restriction in the case \(p=2\).

\section{Preliminaries} \label{Preliminaries}

\subsection{Basics on skew braces}

We recall some basic definitions and properties of skew braces that
will be used throughout the paper. Let $(B,+,\cdot)$ be a skew brace as in the introduction.
The identity elements of $(B,+)$ and $(B,\cdot)$ coincide and will be denoted by
\(0\). For every \(a\in B\), define $\lambda_a:B\longrightarrow B$ by 
$\lambda_a(b)=-a+a\cdot b.$
The skew brace identity implies that $\lambda_a\in\operatorname{Aut}(B,+)$
for every \(a\in B\), and that the map $\lambda:(B,\cdot)\longrightarrow\operatorname{Aut}(B,+)$ given by $\lambda(a)=\lambda_{a}$
is a group homomorphism. Moreover, the lambda map allows one to express each of the two group
operations in terms of the other. Indeed, for all \(a,b\in B\),
\[
a\cdot b=a+\lambda_a(b)
\qquad\text{and}\qquad
a+b=a\cdot\lambda_a^{-1}(b).
\]
It is also convenient to introduce the operation
$a*b=\lambda_a(b)-b$
for all \(a,b\in B\). If \(H\) and \(K\) are subsets of \(B\), we denote by $H*K=\bigl\langle h*k \mid h\in H,\ k\in K\bigr\rangle_+$
the additive subgroup of \((B,+)\) generated by all elements \(h*k\), with
\(h\in H\) and \(k\in K\).
A subset \(C\subseteq B\) is called a \emph{subbrace} of \(B\) if $(C,+)\leq(B,+)$ and $(C,\cdot)\leq(B,\cdot).$ 
A subgroup \(I\leq(B,+)\) is called a \emph{left ideal} of \(B\) if it
is invariant under the lambda action, that is,
$\lambda_b(I)=I$
for every \(b\in B\). We record the following useful observation.

\begin{Proposition}\label{lem:characteristic-left-ideal}
Let \(B\) be a skew brace and let
\(H\) be a left ideal of \(B\). Then $H\leq(B,\cdot).$
\end{Proposition}

\begin{proof}
If \(h,k\in H\), then $h\cdot k=h+\lambda_h(k)\in H$ and  $h_{\cdot}^{-1}=\lambda_h^{-1}(-h)\in H$ and thus \(H\leq(B,\cdot)\).
\end{proof}

A subset \(I\subseteq B\) is called an \emph{ideal} of \(B\), and we write $I \trianglelefteq B$, if $I$ is a left ideal which is normal in both the additive and
multiplicative groups.

\begin{Proposition}\label{lem:left-ideal-criterion} Let \(B\) be a skew brace and let \(I\) be a left ideal of \(B\) such that \(I\trianglelefteq(B,+)\). Then \(I\) is an ideal of \(B\) if and only if $I*B\leq I.$
\end{Proposition}

\begin{proof}
Suppose first that \(I\) is an ideal of \(B\). Let \(i\in I\) and
\(b\in B\). Since \(I\trianglelefteq(B,\cdot)\), we have
\(i\cdot b\in b\cdot I\). Moreover, since \(I\) is a left ideal we have $b\cdot I=b+\lambda_b(I)=b+I=I+b,$
where the last equality follows from \(I\trianglelefteq(B,+)\).
Hence $i+\lambda_i(b)=i\cdot b\in I+b.$
Since \(i\in I\), it follows that \(\lambda_i(b)\in I+b\), and therefore $i*b=\lambda_i(b)-b\in I.$
Thus \(I*B\leq I\). Conversely, suppose that \(I*B\leq I\). For \(i\in I\) and \(b\in B\),
we have \(\lambda_i(b)=i*b+b\in I+b\), and hence
\[
i\cdot b=i+\lambda_i(b)\in I+b=b+I=b\cdot I.
\]
Therefore \(I\cdot b\subseteq b\cdot I\). Since both sets have
cardinality \(|I|\), we obtain \(I\cdot b=b\cdot I\) for every
\(b\in B\). Thus \(I\trianglelefteq(B,\cdot)\), and hence \(I\) is an
ideal of \(B\).
\end{proof}

If \(I\) is an ideal of \(B\), then the quotient set \(B/I\) inherits
a skew brace structure defined by
\[
(a+I)+(b+I)=a+b+I
\qquad\text{and}\qquad
(a+I)\cdot(b+I)=a\cdot b+I.
\]
A skew brace \(B\) is called \emph{simple} if $|B|>1$ and its only ideals are \(0\) and \(B\).  The \emph{socle} of a skew brace \(B\) is defined by $\operatorname{Soc}(B)
=
\ker\lambda\cap Z(B,+),$
where $\ker\lambda
=
\{a\in B:\lambda_a=\operatorname{id}_B\}.$
In particular, the additive and multiplicative operations coincide on
\(\operatorname{Soc}(B)\).
It is well known that \(\operatorname{Soc}(B)\) is an ideal of \(B\). The following Proposition will be useful later.

\begin{Proposition}\label{prop:subbrace-socle-ideal}
Let \(B\) be a skew brace and let \(C\leq\operatorname{Soc}(B)\). Then
\(C\) is an ideal of \(B\) if and only if \(C\) is invariant under the
lambda action, that is, $\lambda_b(C)=C$
for every \(b\in B\).
\end{Proposition}

\begin{proof}
Suppose first that \(C\) is invariant under the lambda action. Since $C\leq\operatorname{Soc}(B)\leq Z(B,+),$
we have $C\trianglelefteq(B,+).$ Let \(c\in C\) and \(b\in B\). Since \(c\in\operatorname{Soc}(B)\), we
have \(\lambda_c=\operatorname{id}_B\). Hence
$b\cdot c\cdot b_{\cdot}^{-1}=\lambda_b(c).$
By the lambda invariance of \(C\), it follows that $b\cdot c\cdot b_{\cdot}^{-1}\in C.$
Therefore $C\trianglelefteq(B,\cdot).$
Thus \(C\) is an ideal of \(B\).
Conversely, if \(C\) is an ideal of \(B\), then it is invariant under
the lambda action by definition.
\end{proof}

\subsection{The generalized lambda-map}

We first introduce an important tool, called the generalized lambda map, which provides a useful method for detecting ideals in a skew brace. Let \(B\) be a skew brace, and let \(H\) be a characteristic subgroup of
\((B,+)\). The
lambda map of \(B\) induces a homomorphism
\[
\lambda_H:(B,\cdot)\longrightarrow \Aut((B,+)/H),
\qquad
b\longmapsto \bigl(a+H\longmapsto \lambda_b(a)+H\bigr).
\]
When \(H=0\), the map \(\lambda_H\) coincides with the usual lambda map
\(\lambda\). We shall call \(\lambda_H\) the generalized lambda map associated
with \(H\). The following lemma provides a useful criterion for proving that a characteristic
subgroup of the additive group is an ideal of \(B\).

\begin{Proposition}\label{lemmaKer}
Let \(B\) be a finite skew brace and let \(H\) be a characteristic
subgroup of \((B,+)\). Then \(H\leq\ker(\lambda_H)\) if and only if
\(H\) is an ideal of \(B\).
\end{Proposition}

\begin{proof}
Suppose first that $H \le \ker(\lambda_H)$. Let $h \in H$ and $b \in B$. Since $h \in \ker(\lambda_H)$, we have $\lambda_h(b)+H = b+H,$
which implies $\lambda_h(b)-b \in H$. Therefore $h * b = \lambda_h(b)-b \in H,$ and hence $H * B \subseteq H$. Thus, by Proposition \ref{lem:left-ideal-criterion}, $H \trianglelefteq B$. Conversely, assume that $H \trianglelefteq B$. Then, for all $h \in H$ and $b \in B$, we have $\lambda_h(b)-b = h * b \in H,$ which implies $\lambda_h(b)+H = b+H$. Hence $h \in \ker(\lambda_H)$, and therefore $H \le \ker(\lambda_H)$.
\end{proof}

As an application we have the following result:

\begin{Proposition}\label{IdealCoprimeAut}
Let \(B\) be a finite skew brace and let \(H\) be a characteristic subgroup of
\((B,+)\). Suppose that $\gcd\bigl(|H|,|\operatorname{Aut}((B,+)/H)|\bigr)=1.$
Then \(H\) is an ideal of \(B\).
\end{Proposition}

\begin{proof}
Consider the generalized lambda map $\lambda_H:(B,\cdot)\longrightarrow \operatorname{Aut}((B,+)/H).$ We claim that $H\leq \ker(\lambda_H).$ Indeed, the image \(\lambda_H(H)\) has order dividing \(|H|\), since
\((H,\cdot)\) is a subgroup of \((B,\cdot)\) of order \(|H|\). On the other
hand, \(\lambda_H(H)\) is a subgroup of $\operatorname{Aut}((B,+)/H).$
Therefore
\[
|\lambda_H(H)|
\mid
\gcd\bigl(|H|,|\operatorname{Aut}((B,+)/H)|\bigr)=1.
\]
Hence $\lambda_H(H)=1$ and we obtain $H\leq \ker(\lambda_H).$ By Proposition \ref{lemmaKer} we deduce that $H$ is an ideal of $B$.
\end{proof}

\subsection{$Z$-groups and the holomorph}

Recall that a finite group \(G\) is called a \(Z\)-group if all its Sylow subgroups are cyclic. The next theorem gives a full description of such groups:

\begin{Theorem}[{\cite[Theorem~9.4.3]{Hall1959}}] \label{Zgroups}
Let \(G\) be a finite \(Z\)-group. Then \(G\) admits a presentation of the form
\[
G=\langle a,b \mid a^{m}=b^{n}=1,\; bab^{-1}=a^{r}\rangle,
\]
where
\[
(m,n)=(m,r-1)=1
\qquad\text{and}\qquad
r^{n}\equiv 1 \pmod{m}.
\]
\end{Theorem}

Let \(H\) be a group. The \emph{holomorph} of \(H\) is the semidirect
product $\operatorname{Hol}(H)=H\rtimes\operatorname{Aut}(H),$
where \(\operatorname{Aut}(H)\) acts on \(H\) by evaluation.The holomorph acts on \(H\) by
\[
(h,\alpha)\cdot x\longmapsto h+\alpha(x).
\]
A subgroup \(G\leq\operatorname{Hol}(H)\) is called \emph{regular} if
this action is regular, that is, if for every \(x,y\in H\) there exists
a unique \(g\in G\) such that \(g \cdot x=y\). 
The following result describes the groups of \(2\)-power order whose
holomorph contains an element of maximal possible order.

\begin{Theorem}[{\cite[Theorem~6.1]{Byott2007}}]
\label{thm:Byott-cyclic-element}
Let \(H\) be a group of order \(2^n\), where \(n\geq3\). If
\(\operatorname{Hol}(H)\) contains an element of order \(2^n\), then
\(H\) is cyclic, dihedral, or generalized quaternion. More precisely,
\[
H\simeq C_{2^n},\qquad
H\simeq D_{2^n},
\qquad\text{or}\qquad
H\simeq Q_{2^n}.
\]
\end{Theorem}

The odd-order case was treated by Kohl in \cite[Theorem~4.5]{Kohl1998}.

\begin{Theorem}[{\cite[Theorem~4.5]{Kohl1998}}]
\label{thm:Kohl-cyclic-element}
Let \(p\) be an odd prime and let \(H\) be a group of order \(p^n\).
If \(\operatorname{Hol}(H)\) contains an element of order \(p^n\),
then \(H\) is cyclic.
\end{Theorem}

The connection between skew braces and regular subgroups of holomorphs
is given by the following fundamental result.

\begin{Theorem}[{\cite[Theorem~4.2]{Guarnieri_2016}}]
\label{Connection}
Let \(H\) be a group and let \(G\) be a regular subgroup of
\(\operatorname{Hol}(H)\). Then there exists a skew brace
\((B,+,\cdot)\) such that $(B,+)\simeq H$ and $(B,\cdot)\simeq G.$ 
Conversely, if \((B,+,\cdot)\) is a skew brace, then \((B,\cdot)\) can
be identified with a regular subgroup of
\(\operatorname{Hol}(B,+)\) via the injective homomorphism $(B,\cdot)\longrightarrow\operatorname{Hol}(B,+), $
$b\longmapsto (b,\lambda_b).$
\end{Theorem}

Let \(B\) be a finite skew brace and let \(p\) be a prime divisor of
\(|B|\). A subbrace \(P\) of \(B\) is called a \emph{Sylow
\(p\)-subbrace} if $(P,+)\in\operatorname{Syl}_p((B,+)).$ and $(P,\cdot)\in\operatorname{Syl}_p((B,\cdot)).$
Equivalently, \(P\) is a subbrace of \(B\) whose order is equal to the
largest power of \(p\) dividing \(|B|\). Such Sylow $p$-subbraces always exist by \cite[Theorem~2.1]{TrumanSylow}.

The next lemma shows that the assumption that the multiplicative group of a skew brace is a $Z$-group imposes strong restrictions on the structure of its additive group. 
Throughout the paper, \(D_{2^a}\) and \(Q_{2^a}\) denote, respectively, the dihedral group and the generalized quaternion group of order \(2^a\).

\begin{Proposition}\label{lem:additive-almost-Sylow-cyclic}
Let \(B\) be a finite skew brace such that \((B,\cdot)\) is a
\(Z\)-group. Then every Sylow subgroup of odd order of \((B,+)\) is
cyclic. Moreover, a Sylow \(2\)-subgroup of \((B,+)\) is cyclic,
dihedral, generalized quaternion, or isomorphic to
\(C_2\times C_2\).
\end{Proposition}

\begin{proof}
Let \(q\) be an odd prime dividing \(|B|\). By
\cite[Theorem~2.1]{TrumanSylow}, the skew brace \(B\) has a Sylow
\(q\)-subbrace \(Q\). Hence $(Q,\cdot)\in\operatorname{Syl}_q(B,\cdot)$ and $(Q,+)\in\operatorname{Syl}_q(B,+).$
Since \((B,\cdot)\) is a \(Z\)-group, the group \((Q,\cdot)\) is
cyclic. By Theorem~\ref{Connection}, \((Q,\cdot)\) can be identified
with a regular subgroup of \(\operatorname{Hol}(Q,+)\). In particular,
\(\operatorname{Hol}(Q,+)\) contains an element of order \(|Q|\).
Since \(q\) is odd, Theorem~\ref{thm:Kohl-cyclic-element} implies that
\((Q,+)\) is cyclic. Let now \(P\) be a Sylow \(2\)-subbrace of \(B\). Then $(P,\cdot)\in\operatorname{Syl}_2(B,\cdot),$
and hence \((P,\cdot)\) is cyclic. By Theorem~\ref{Connection}, the
group \((P,\cdot)\) can be identified with a regular subgroup of
\(\operatorname{Hol}(P,+)\). Therefore
\(\operatorname{Hol}(P,+)\) contains an element of order \(|P|\). If \(|P|\leq4\), then \((P,+)\) is cyclic or isomorphic to
\(C_2\times C_2\). Suppose that \(|P|=2^a\) with \(a\geq3\). By
Theorem~\ref{thm:Byott-cyclic-element}, we have $(P,+)\simeq C_{2^a}$, $(P,+)\simeq D_{2^a}$ or $(P,+)\simeq Q_{2^a}$
Since \((P,+)\in\operatorname{Syl}_2(B,+)\), the conclusion follows.
\end{proof}

\subsection{Supersoluble groups}

Recall that a finite group \(G\) is called \emph{supersoluble} if it admits a normal series $ 1=G_0\trianglelefteq G_1\trianglelefteq\cdots \trianglelefteq G_n=G$ such that every factor \(G_i/G_{i-1}\) is cyclic of prime order. The following lemma is well known. Nevertheless, for the sake of clarity and
to keep the exposition self-contained, we include a short proof.

\begin{Lemma}\label{ZGroupsSupersoluble}
Every finite \(Z\)-group is supersoluble.
\end{Lemma}

\begin{proof}
Let \(G\) be a finite \(Z\)-group. By Theorem \ref{Zgroups} both the derived
subgroup \(G'\) and the quotient \(G/G'\) are cyclic. Since \(G'\) is a finite cyclic group, there exists a series $1=N_0\triangleleft N_1\triangleleft\cdots\triangleleft N_r=G'$
such that every factor \(N_i/N_{i-1}\) has prime order. Each \(N_i\) is
characteristic in \(G'\), because a finite cyclic group has a unique subgroup
of each order. Since \(G'\) is characteristic in \(G\), it follows that
\(N_i\triangleleft G\) for every \(i\).
Similarly, since \(G/G'\) is cyclic, it has a series $1=\overline{M}_0\triangleleft\overline{M}_1
\triangleleft\cdots\triangleleft\overline{M}_s=G/G'$
whose factors have prime order. Let \(M_i\) be the inverse image of
\(\overline{M}_i\) under the canonical projection $G\longrightarrow G/G'.$
Then every \(M_i\) is normal in \(G\), and
$M_i/M_{i-1}\simeq
\overline{M}_i/\overline{M}_{i-1}$
has prime order.
Combining the two series, we obtain a normal series $1=N_0\triangleleft\cdots\triangleleft N_r=G'
=M_0\triangleleft M_1\triangleleft\cdots\triangleleft M_s=G$
all of whose factors have prime order. Therefore \(G\) is supersoluble.
\end{proof}

We also recall the following well-known properties of finite supersoluble groups.

\begin{Theorem}[{\cite[Theorem~5.4.8]{Robinson}}]
\label{thm:supersoluble-Hall-Sylow}
Let \(G\) be a finite supersoluble group. If \(p\) is the smallest
prime divisor of \(|G|\), then \(G\) has a normal Hall \(p'\)-subgroup.
Moreover, if \(q\) is the largest prime divisor of \(|G|\), then every
Sylow \(q\)-subgroup of \(G\) is normal.
\end{Theorem}

\section{Simple skew braces whose multiplicative group is a $Z$-group} \label{Section simple multiplicative Z group}

In this section, we classify finite simple skew braces whose multiplicative group is a $Z$-group. We treat separately the cases in which the additive group is soluble and insoluble.

\subsection{The soluble additive case}

We first describe the small \(2\)-subgroups of a \(Z\)-group.

\begin{Lemma}\label{lem:small-2-subgroups-general-Z}
Let \(G\) be a finite \(Z\)-group of even order. Then
\(G=N\rtimes P\), where \(N\) is the normal Hall \(2'\)-subgroup of \(G\)
and \(P=\langle y\rangle\simeq C_{2^n}\) is a Sylow \(2\)-subgroup of
\(G\). Let
\[
2^s=
\left|
\operatorname{Im}\bigl(P\longrightarrow\operatorname{Aut}(N)\bigr)
\right|.
\]
For every integer \(k\) satisfying \(0\leq k\leq n-s\), the group \(G\)
has a unique subgroup of order \(2^k\), namely
\(\langle y^{2^{n-k}}\rangle\). In particular, this subgroup is central
in \(G\).
\end{Lemma}

\begin{proof}
By Lemma \ref{ZGroupsSupersoluble} \(G\) is supersoluble. Since \(2\) is the smallest prime divisor of \(|G|\),
Theorem~\ref{thm:supersoluble-Hall-Sylow} yields a normal Hall
\(2'\)-subgroup \(N\) of \(G\). Consequently $G=N\rtimes P$
for every \(P\in\operatorname{Syl}_2(G)\). Let \(0\leq k\leq n-s\), and put
\(U_k=\langle y^{2^{n-k}}\rangle\). Since \(2^{n-k}\) is divisible by
\(2^s\), the subgroup \(U_k\) is contained in the kernel of the action of
\(P\) on \(N\). Therefore \(U_k\leq C_G(N)\). Since \(P\) is abelian,
\(U_k\) also centralizes \(P\), and hence \(U_k\leq Z(G)\). Every Sylow \(2\)-subgroup of \(G\) is conjugate to \(P\), and every
\(2\)-subgroup of \(G\) is contained in a Sylow \(2\)-subgroup. Since
\(U_k\) is central, every Sylow \(2\)-subgroup contains \(U_k\).
Moreover, each Sylow \(2\)-subgroup is cyclic and therefore has a unique
subgroup of order \(2^k\). Consequently, \(U_k\) is the unique subgroup
of order \(2^k\) in \(G\).
\end{proof}

\begin{Lemma}\label{lem:characteristic-2-subgroup-ideal}
Let \(B\) be a finite skew brace such that \((B,\cdot)\) is a \(Z\)-group
of even order. Write
\((B,\cdot)=N\rtimes P\), where \(N\) is its normal Hall
\(2'\)-subgroup and \(P\simeq C_{2^n}\). Let
\[
2^s=
\left|
\operatorname{Im}\bigl(P\longrightarrow\operatorname{Aut}(N)\bigr)
\right|.
\]
Suppose that \(K\) is a non-trivial characteristic \(2\)-subgroup of
\((B,+)\) such that \(|K|=2^k\) and \(k\leq n-s\). Then \(K\) is an
ideal of \(B\).
\end{Lemma}

\begin{proof}
Since \(K\operatorname{char}(B,+)\), the subgroup \(K\) is normal in
\((B,+)\) and invariant under the lambda action. Hence \(K\) is a left
ideal of \(B\), and therefore, by Proposition \ref{lem:characteristic-left-ideal}, we conclude that \(K\leq(B,\cdot)\). By Lemma~\ref{lem:small-2-subgroups-general-Z}, the group \((B,\cdot)\)
has a unique subgroup of order \(2^k\). Hence
\(K\trianglelefteq(B,\cdot)\). Therefore \(K\) is an ideal of \(B\).
\end{proof}

\begin{Lemma}\label{lem:automorphism-groups-small-2-groups}
Let \(P\) be a finite \(2\)-group which is cyclic, dihedral,
generalized quaternion, or isomorphic to \(C_2\times C_2\). The following statements hold:
\begin{enumerate}
    \item \(\operatorname{Aut}(C_{2^n})\) is a \(2\)-group for every
    \(n\geq1\);

    \item \(\operatorname{Aut}(D_{2^n})\) is a \(2\)-group for every
    \(n\geq3\);

    \item \(\operatorname{Aut}(Q_{2^n})\) is a \(2\)-group for every
    \(n\geq4\);

    \item
    \(\operatorname{Aut}(Q_8)\simeq S_4\) and
    \(\operatorname{Aut}(C_2\times C_2)\simeq S_3\).
\end{enumerate}
\end{Lemma}

\begin{proof}
These facts follow immediately from the standard presentations of cyclic, dihedral and generalized quaternion groups; see, for instance, \cite[Kapitel~I]{Huppert1967}. The exceptional isomorphisms \(\operatorname{Aut}(Q_8)\simeq S_4\) and \(\operatorname{Aut}(C_2\times C_2)\simeq S_3\) are standard.
\end{proof}

We can now prove the main result of this section. We denote by \(S_{12,22}\) and \(S_{12,23}\) the two simple skew braces
of order \(12\) appearing in the classification of skew braces of small
order in \cite{KonovalovSmoktunowiczVendramin}. Both have additive group
isomorphic to \(A_4\) and multiplicative group isomorphic to the
non-abelian semidirect product \(C_3\rtimes C_4\).

\begin{proof}[Proof of Theorem \ref{thm:A}]

Assume that $|B|$ has not prime order. We divide the proof in several steps.

\medskip

\noindent\textbf{Step 1}: \textit{The order of \(B\) is even and \((B,\cdot)=N\rtimes P\).}

\medskip 

Suppose, to the contrary, that
\(|B|\) is odd, and let \(p\) be the smallest prime divisor of \(|B|\).
By \cite[Theorem C]{DameleZGroups} $B$ is supersoluble. Thus \(B\) admits a Hall \(p'\)-ideal \(H\). Since \(B\) is simple and
\(H\neq B\), we have \(H=0\). Thus \(|B|\) is a power of
\(p\). By \cite[Proposition~4.4]{CedoSmoktunowiczVendramin2019} a simple skew brace of
prime-power order has order \(p\). Therefore \(B\simeq C_p\),
contradicting the assumption that \(B\) has not prime order. Consequently,
\(|B|\) is even. By Lemma~\ref{lem:small-2-subgroups-general-Z}, we may write
\((B,\cdot)=N\rtimes P\), where \(N\) is the normal Hall
\(2'\)-subgroup of \((B,\cdot)\) and
\(P=\langle y\rangle\simeq C_{2^n}\), with \(n\geq1\). Put $2^s=
\left|
\operatorname{Im}\bigl(P\longrightarrow\operatorname{Aut}(N)\bigr)
\right|.$

\medskip 

\noindent\textbf{Step 2}: \textit{\(O_{2'}(B,+)=1\).}

\medskip 

Suppose that \(O_{2'}(B,+)\neq 1\). Then there exists an odd prime \(q\) such that
\(O_q(B,+)\neq 1\). Choose \(q\) maximal with this property. By Proposition \ref{lem:additive-almost-Sylow-cyclic} \(O_q(B,+)\) is cyclic. Let \(M\) be its
unique subgroup of order \(q\). Then
\(M\operatorname{char}O_q(B,+)\), while
\(O_q(B,+)\operatorname{char}(B,+)\), and therefore
\(M\operatorname{char}(B,+)\). It follows that \(M\trianglelefteq(B,+)\) and that \(M\) is invariant
under the lambda action. Thus \(M\) is a left ideal of \(B\). In
particular, by Proposition \ref{lem:characteristic-left-ideal}, \(M\) is a subgroup of \((B,\cdot)\).
We claim that \(q\) is the largest prime divisor of \(|B|\). Suppose, to
the contrary, that a prime \(r>q\) divides \(|B|\). By the maximality of
\(q\), we have \(O_r(B,+)=1\), and hence
\(r\nmid |F(B,+)|\). Since \((B,+)\) is soluble, the Fitting centralizer theorem \cite[Theorem~5.4.4(ii)]{Robinson} yields
\(C_{(B,+)}(F(B,+))\leq F(B,+)\). Therefore
$(B,+)/C_{(B,+)}(F(B,+))$
embeds into \(\operatorname{Aut}(F(B,+))\). Since
\(r\nmid |F(B,+)|\), it follows that
\(r\mid |\operatorname{Aut}(F(B,+))|\). We have
\[
F(B,+)=O_2(B,+)\times
\prod_{\substack{O_t(B,+)\neq 1\\ t\text{ odd}}}O_t(B,+).
\]
The factors in this direct product have pairwise coprime orders and are
characteristic in \(F(B,+)\). Consequently,
\[
\operatorname{Aut}(F(B,+))
\simeq
\operatorname{Aut}(O_2(B,+))
\times
\prod_{\substack{O_t(B,+)\neq 1\\ t\text{ odd}}}
\operatorname{Aut}(O_t(B,+)).
\]
For every odd prime \(t\) such that \(O_t(B,+)\neq 1\), the group
\(O_t(B,+)\) is cyclic. Hence every prime divisor of
\(|\operatorname{Aut}(O_t(B,+))|\) is at most \(t\), and therefore at
most \(q\). Moreover, \(O_2(B,+)\) is contained in a Sylow \(2\)-subgroup of
\((B,+)\). By Proposition \ref{lem:additive-almost-Sylow-cyclic} \(O_2(B,+)\) is  cyclic, dihedral, generalized quaternion, or isomorphic to a subgroup of \(C_2\times C_2\). Consequently, by Lemma \ref{lem:automorphism-groups-small-2-groups}, the only possible odd prime divisor of \(|\operatorname{Aut}(O_2(B,+))|\) is \(3\). Thus every prime divisor of
\(|\operatorname{Aut}(F(B,+))|\) is at most \(q\), contradicting
\(r>q\). Therefore \(q\) is the largest prime divisor of \(|B|\). By
Lemma \ref{ZGroupsSupersoluble} \((B,\cdot)\) is supersoluble. Since \(q\) is the largest prime divisor of \(|B|\),
Theorem~\ref{thm:supersoluble-Hall-Sylow} implies that every Sylow
\(q\)-subgroup of \((B,\cdot)\) is normal. Hence, if
\(Q\in \operatorname{Syl}_q(B,\cdot)\), then $Q\trianglelefteq (B,\cdot).$ Moreover, \(Q\) is cyclic. Since \(M\) is a
multiplicative subgroup of order \(q\), we have \(M\leq Q\). The group
\(Q\) has a unique subgroup of order \(q\), and therefore
\(M\operatorname{char}Q\). Since \(Q\trianglelefteq(B,\cdot)\), it
follows that \(M\trianglelefteq(B,\cdot)\).
Thus \(M\) is an ideal of \(B\). Since \(|M|=q\) and \(B\) is not of
prime order, \(M\) is non-trivial and proper.
This contradiction proves that \(O_{2'}(B,+)=1\).

\medskip

\noindent\textbf{Step 3}: \textit{A small characteristic \(2\)-subgroup of \((B,+)\).}

\medskip

Since \((B,+)\) is soluble, \(F(B,+)\neq1\). Moreover, all its normal
Sylow subgroups of odd order are trivial by Step~2. Hence
\(F(B,+)=O_2(B,+)\neq1\).
By Proposition~\ref{lem:additive-almost-Sylow-cyclic}, the \(2\)-rank of
\(O_2(B,+)\), that is, the largest integer \(r\) such that
\(O_2(B,+)\) contains an elementary abelian subgroup of order \(2^r\),
is at most \(2\). Set
\[
K=\Omega_1\bigl(Z(O_2(B,+))\bigr),
\]
where \(\Omega_1\bigl(Z(O_2(B,+))\bigr)\) denotes the subgroup generated
by the elements of order \(2\) in \(Z(O_2(B,+))\). Then \(K\neq1\), \(K\operatorname{char}(B,+)\), and $|K|\leq4.$

\medskip 

\noindent\textbf{Step 4}: \textit{The case \(n-s\geq2\) cannot occur.} 

\smallskip

Suppose  that \(n-s\geq2\). Then
\(v_2(|K|)\leq2\leq n-s\). By
Lemma~\ref{lem:characteristic-2-subgroup-ideal}, \(K\) is a non-trivial
proper ideal of \(B\), a contradiction. Therefore \(n-s\in\{0,1\}\).

\medskip 

\noindent\textbf{Step 5}: \textit{The case \(n-s=0\) cannot occur.} 

\smallskip

Suppose that \(n-s=0\). Recall that $F(B,+)=O_2(B,+).$ By the Fitting centralizer theorem
\cite[Theorem~5.4.4(ii)]{Robinson}, we have $C_{(B,+)}\bigl(F(B,+)\bigr)\leq F(B,+).$
Since \(F(B,+)\operatorname{char}(B,+)\), conjugation induces a
homomorphism $(B,+)\longrightarrow\operatorname{Aut}\bigl(F(B,+)\bigr)$
whose kernel is $C_{(B,+)}\bigl(F(B,+)\bigr).$
Consequently $(B,+)/C_{(B,+)}\bigl(F(B,+)\bigr)
\hookrightarrow
\operatorname{Aut}\bigl(F(B,+)\bigr).$
As $C_{(B,+)}\bigl(F(B,+)\bigr)\leq F(B,+),$
it follows that $\bigl|(B,+)/F(B,+)\bigr|
\mid
\bigl|(B,+)/C_{(B,+)}(F(B,+))\bigr|,$
and hence $\bigl|(B,+)/F(B,+)\bigr|
\mid
\left|\operatorname{Aut}\bigl(F(B,+)\bigr)\right|.$ By Proposition~\ref{lem:additive-almost-Sylow-cyclic}, the group
\(F(B,+)\) is contained in a Sylow \(2\)-subgroup of \((B,+)\) which
is cyclic, dihedral, generalized quaternion, or isomorphic to
\(C_2\times C_2\). Therefore \(F(B,+)\) is itself cyclic, dihedral,
generalized quaternion, or elementary abelian of order at most \(4\). By Lemma \ref{lem:automorphism-groups-small-2-groups} it follows that $\left|
\operatorname{Aut}\bigl(F(B,+)\bigr)
\right|_{2'}
\leq 3,$
where, for a positive integer \(m\), we denote by \(m_{2'}\) the largest
odd divisor of \(m\).
Since \(F(B,+)\) is a \(2\)-group, the odd part of
\(\lvert(B,+)/F(B,+)\rvert\) is equal to the odd part of \(|B|\),
namely \(|N|\). Hence $|N|\mid
\left|
\operatorname{Aut}\bigl(F(B,+)\bigr)
\right|_{2'},$
and therefore $|N|\mid3.$
Since \(N\neq1\), we obtain $N\simeq C_3.$
The assumption \(n-s=0\) means that the action
$P\longrightarrow\operatorname{Aut}(N)$
is faithful. Since $P\simeq C_{2^n},$ $N\simeq C_3$ and $\operatorname{Aut}(C_3)\simeq C_2,$
the cyclic group \(C_{2^n}\) embeds into \(C_2\). Thus \(n=1\), and
hence $|B|=|N||P|=3\cdot2=6.$
By \cite[Proposition~3.6]{KonovalovSmoktunowiczVendramin}, no skew
brace of order \(6\) is simple. This contradiction shows that the
case \(n-s=0\) cannot occur.

\medskip 

\noindent\textbf{Step 6}: \textit{The case \(n-s=1\) and the conclusion of the proof.} 

\smallskip

We must have \(n-s=1\). If \(|K|=2\), then
Lemma~\ref{lem:characteristic-2-subgroup-ideal} again implies that \(K\)
is an ideal of \(B\), a contradiction. Hence \(|K|=4\).
Since \(O_2(B,+)\) is contained in a cyclic, dihedral, generalized
quaternion, or elementary abelian Sylow \(2\)-subgroup, the equality $|\Omega_1(Z(O_2(B,+)))|=4$
forces \(O_2(B,+)\simeq C_2\times C_2\). Thus
\(F(B,+)=O_2(B,+)\simeq C_2\times C_2\).
By the Fitting centralizer theorem \cite[Theorem~5.4.4(ii)]{Robinson} we have $C_{(B,+)}(F(B,+))\leq F(B,+).$
Since \(F(B,+)\leq C_{(B,+)}(F(B,+))\), we have
$C_{(B,+)}(F(B,+))=F(B,+).$
Therefore $(B,+)/F(B,+)$
embeds into $\operatorname{Aut}(F(B,+))
\simeq\operatorname{GL}_2(2)\simeq S_3.$
Consequently we have $|(B,+)/F(B,+)|\mid6.$
Since \(|B|=|N|2^n\) and \(|F(B,+)|=4\), we obtain $|N|2^{n-2}\mid6.$ As \(|N|>1\) is odd, it follows that \(|N|=3\) and \(n\leq3\).
If \(n=3\), then \(n-s=1\) gives \(s=2\). This is impossible, since the
action of \(P\simeq C_8\) on \(N\simeq C_3\) would have image of order
\(4\), whereas \(\operatorname{Aut}(C_3)\simeq C_2\). Therefore \(n=2\).
It follows that \(|N|=3\) and \(s=1\). Thus
\((B,\cdot)\simeq C_3\rtimes C_4\), where the action is non-trivial. Thus \(B\) is a simple skew brace of order \(12\). It follows from
\cite[Proposition~3.6]{KonovalovSmoktunowiczVendramin} that $B\simeq S_{12,22}$ or $B\simeq S_{12,23}.$

\end{proof}

\subsection{The insoluble additive case}

In this subsection, we determine the additive group of a finite simple skew
brace whose multiplicative group is a \(Z\)-group, under the assumption
that the additive group is insoluble. We briefly recall the definitions
and the properties of the classical groups of degree \(2\) that will
be used in the proof. 

Let \(p\) be an odd prime. The special linear group, the projective special linear group and the projective general linear group are defined, respectively, by

\[
\begin{aligned}
\operatorname{SL}_2(p) &= \{A\in\operatorname{GL}_2(p):\det(A)=1\},\\ \operatorname{PSL}_2(p) &= \operatorname{SL}_2(p)/Z(\operatorname{SL}_2(p)),\\ \operatorname{PGL}_2(p) &= \operatorname{GL}_2(p)/Z(\operatorname{GL}_2(p)).
\end{aligned}
\]
We refer to \cite[Kapitel~II, \S8]{Huppert1967} for the standard
properties of these groups.

\begin{Lemma}\label{lem:linear-groups-degree-two}
Let \(p\geq5\) be a prime, and write $2^a=|\operatorname{SL}_2(p)|_2.$
Then the following statements hold:
\begin{enumerate}
    \item $ |\operatorname{SL}_2(p)|
    =
    |\operatorname{PGL}_2(p)|
    =
    p(p^2-1),$
    and therefore $|\operatorname{SL}_2(p)|_2
    =
    |\operatorname{PGL}_2(p)|_2
    =
    2^a;$

    \item $\operatorname{Aut}(\operatorname{PSL}_2(p))
\simeq
\operatorname{Aut}(\operatorname{SL}_2(p))
\simeq
\operatorname{PGL}_2(p);$
    \item a Sylow \(2\)-subgroup of
    \(\operatorname{PGL}_2(p)\) is dihedral of order \(2^a\).
    Consequently, every cyclic \(2\)-subgroup of
    \(\operatorname{PGL}_2(p)\) has order at most \(2^{a-1}\).
\end{enumerate}
\end{Lemma}

\begin{proof}
These are standard properties of the finite linear groups of degree
\(2\); see \cite[Kapitel~II, \S8]{Huppert1967}. 
\end{proof}

We shall use the following classification theorem of Suzuki and Wong.

\begin{Theorem}[{\cite[Theorem~2]{WongSemidihedral}}]
\label{thm:Suzuki-Wong}
Let \(G\) be a finite insoluble group such that every Sylow subgroup of odd order is cyclic and a Sylow \(2\)-subgroup has a cyclic subgroup of index at most \(2\). Then \(G\) has a normal subgroup \(G_1\) satisfying $[G:G_1]\leq2$ and $G_1=L\times M,$ where $L\simeq\operatorname{PSL}_2(p)$ or $L\simeq\operatorname{SL}_2(p)$ for some prime \(p\geq5\), while \(M\) is a \(Z\)-group such that $\gcd(|L|,|M|)=1.$
\end{Theorem}

The next lemma strengthens the conclusion of
Theorem~\ref{thm:Suzuki-Wong} in the form needed for skew braces.

\begin{Lemma}\label{lem:Suzuki-Wong-characteristic-factor}
Let \(G\) be as in Theorem~\ref{thm:Suzuki-Wong}, and write $G_1=L\times M$
as in that theorem. Then \(M\operatorname{char}G\) and $\gcd\left(|M|,\left|\operatorname{Aut}(G/M)\right|\right)=1.$
\end{Lemma}

\begin{proof}
Let \(\pi=\pi(M)\) be the set of prime divisors of \(|M|\). We divide
the proof into two steps.

\medskip

\noindent\textbf{Step 1}: \textit{\(M\operatorname{char}G\).}

\medskip

Since \(L\) has even order and \(\gcd(|L|,|M|)=1\), the group \(M\)
has odd order. Moreover, $|G/M|=|L|\,[G:G_1]$
and \([G:G_1]\leq2\). Hence $\gcd(|M|,|G/M|)=1.$ In the direct product \(G_1=L\times M\), the subgroup \(M\) is the
unique Hall \(\pi\)-subgroup. Therefore \(M\operatorname{char}G_1\).
Since \(G_1\trianglelefteq G\), it follows that \(M\trianglelefteq G\). We now show that
\[
M=O_\pi(G).
\]
Indeed, let \(N\) be any normal \(\pi\)-subgroup of \(G\). Its image
\(NM/M\) is a normal \(\pi\)-subgroup of \(G/M\). Since
\(\gcd(|M|,|G/M|)=1\), no prime in \(\pi\) divides \(|G/M|\), and
therefore $NM/M=1.$
Thus \(N\leq M\). Since \(M\) is itself a normal \(\pi\)-subgroup of
\(G\), we conclude that \(M=O_\pi(G)\). Consequently,
\(M\operatorname{char}G\).

\medskip

\noindent\textbf{Step 2}: \textit{the orders of \(M\) and
\(\operatorname{Aut}(G/M)\) are coprime.}

\medskip

Put $H=G/M.$ The image of \(G_1\) in \(H\) is naturally isomorphic to \(L\).
Thus, identifying this image with \(L\), we have $L\trianglelefteq H$ and $[H:L]\leq2.$
Both \(\operatorname{PSL}_2(p)\) and \(\operatorname{SL}_2(p)\) are
perfect for \(p\geq5\). Since \(H/L\) is abelian, we have $H'\leq L.$
On the other hand $L=L'\leq H'.$
Therefore $H'=L,$
and hence \(L\operatorname{char}H\).
Restriction to \(L\) now defines a homomorphism
\[
\rho:\operatorname{Aut}(H)\longrightarrow\operatorname{Aut}(L),
\qquad
\alpha\longmapsto \alpha|_L.
\]
We claim that \(\ker\rho\) is a \(2\)-group. If \(H=L\), then
\(\ker\rho=1\). Suppose therefore that \([H:L]=2\), and fix
\(x\in H\setminus L\). Let \(\alpha\in\ker\rho\). Since
\(H/L\simeq C_2\) and \(\operatorname{Aut}(C_2)=1\), the automorphism
induced by \(\alpha\) on \(H/L\) is trivial. Hence $\alpha(x)=xz$
for some \(z\in L\).
For every \(l\in L\), the element \(xlx^{-1}\) belongs to \(L\).
Since \(\alpha\) fixes \(L\) pointwise, we obtain
\[
xlx^{-1}
=
\alpha(xlx^{-1})
=
xz\,l\,z^{-1}x^{-1}.
\]
Cancelling \(x\) and \(x^{-1}\), we get $zlz^{-1}=l$
for every \(l\in L\). Thus \(z\in Z(L)\). Moreover, \(\alpha\) is
completely determined by \(z\), because \(H=\langle L,x\rangle\).
Consequently $|\ker\rho|\leq |Z(L)|.$
Now $Z(\operatorname{PSL}_2(p))=1$ and $Z(\operatorname{SL}_2(p))\simeq C_2,$
and hence \(|\ker\rho|\leq2\). In particular, \(\ker\rho\) is a
\(2\)-group. For \(p\geq5\), by Lemma \ref{lem:linear-groups-degree-two}, we have
\[
\operatorname{Aut}(\operatorname{PSL}_2(p))
\simeq
\operatorname{Aut}(\operatorname{SL}_2(p))
\simeq
\operatorname{PGL}_2(p).
\]
Moreover,
\[
|\operatorname{PGL}_2(p)|
=
|\operatorname{SL}_2(p)|
=
2|\operatorname{PSL}_2(p)|.
\]
It follows that every odd prime divisor of
\(|\operatorname{Aut}(L)|\) divides \(|L|\). Since \(\ker\rho\) is a
\(2\)-group, every odd prime divisor of
\(|\operatorname{Aut}(H)|\) also divides \(|L|\). Finally, \(|M|\) is odd and \(\gcd(|M|,|L|)=1\). Therefore $\gcd\left(|M|,|\operatorname{Aut}(H)|\right)=1.$
Since \(H=G/M\), the result follows.
\end{proof}

We can now prove Theorem~\ref{thm:B}.

\begin{proof}[Proof of Theorem \ref{thm:B}]
We divide the proof into several steps.

\medskip

\noindent\textbf{Step 1}: \textit{\((B,+)\) has a normal subgroup \(L\) of index at most
\(2\), where $L\simeq\operatorname{PSL}_2(p)$ or $L\simeq\operatorname{SL}_2(p).$}

\medskip

By Proposition~\ref{lem:additive-almost-Sylow-cyclic}, every Sylow subgroup
of odd order of \((B,+)\) is cyclic, while a Sylow \(2\)-subgroup of
\((B,+)\) has a cyclic subgroup of index at most \(2\). Since \((B,+)\) is insoluble, Theorem~\ref{thm:Suzuki-Wong} provides a
normal subgroup \(G_1\) of \((B,+)\) satisfying
$[(B,+):G_1]\leq2$
and $G_1=L\times M,$
where $L\simeq\operatorname{PSL}_2(p)$ or $L\simeq\operatorname{SL}_2(p)$
for some prime \(p\geq5\), while \(M\) is a \(Z\)-group such that $\gcd(|L|,|M|)=1.$
By Lemma~\ref{lem:Suzuki-Wong-characteristic-factor}, the subgroup
\(M\) is characteristic in \((B,+)\).
By Lemma~\ref{lem:Suzuki-Wong-characteristic-factor} we have $\gcd\left(
|M|,
\left|
\operatorname{Aut}\bigl((B,+)/M\bigr)
\right|
\right)=1.$
Therefore, by Proposition~\ref{IdealCoprimeAut}, the
subgroup \(M\) is an ideal of \(B\).
Since \(B\) is simple, either \(M=0\) or \(M=B\). The latter possibility
is impossible because \((B,+)/M\) contains the non-abelian group \(L\).
Thus $M=0.$ Consequently, \((B,+)\) has a normal subgroup \(L\) of index at most
\(2\), where $L\simeq\operatorname{PSL}_2(p)$ or $L\simeq\operatorname{SL}_2(p).$

\medskip

\noindent\textbf{Step 2}: \textit{\((B,+) \simeq \operatorname{PSL}_2(p)\) or \((B,+) \simeq \operatorname{SL}_2(p)\)}

\medskip

Suppose that
$[(B,+):L]=2.$ Since \(L\) is perfect and \((B,+)/L\) is abelian, we have $(B,+)'\leq L$
and $L=L'\leq(B,+)'.$ Thus \(L=(B,+)'\), and hence \(L\operatorname{char}(B,+)\).
Therefore \(L\) is a left ideal of \(B\) and so, by Proposition \ref{lem:characteristic-left-ideal}, \(L\leq(B,\cdot)\).
Since \(L\) has index \(2\) in \((B,\cdot)\), it is normal in
\((B,\cdot)\). Hence \(L\) is an ideal of \(B\), contradicting the
simplicity of \(B\). Therefore \((B,+)=L\).

\medskip

\noindent\textbf{Step 3}: \textit{The case $(B,+)\simeq\operatorname{SL}_2(p)$ cannot happen}

\medskip

Let $K=Z(B,+).$ Since \((B,+)\simeq\operatorname{SL}_2(p)\), we have
$K\simeq C_2.$
Moreover, \(K\operatorname{char}(B,+)\), and hence \(K\) is a left
ideal of \(B\). In particular, by Proposition \ref{lem:characteristic-left-ideal}, $K\leq(B,\cdot).$
Choose \(P\in\operatorname{Syl}_2(B,\cdot)\) such that \(K\leq P\).
Since \((B,\cdot)\) is a \(Z\)-group, \(P\) is cyclic. Write $|P|=2^a.$
Since the additive and multiplicative groups have the same order, by Lemma \ref{lem:linear-groups-degree-two}, we have $2^a
=
|\operatorname{SL}_2(p)|_2
=
|\operatorname{PGL}_2(p)|_2.$
The restriction of the lambda map to \(P\) gives a homomorphism
\[
\lambda|_P:
P\longrightarrow\operatorname{Aut}(B,+)
\simeq\operatorname{Aut}(\operatorname{SL}_2(p))
\simeq\operatorname{PGL}_2(p).
\]
By Lemma \ref{lem:linear-groups-degree-two}, a Sylow \(2\)-subgroup of \(\operatorname{PGL}_2(p)\) is dihedral of
order \(2^a\). Consequently, every cyclic \(2\)-subgroup of
\(\operatorname{PGL}_2(p)\) has order at most \(2^{a-1}\). Since
\(\lambda(P)\) is cyclic, it follows that $|\lambda(P)|\leq2^{a-1}.$
Therefore $|\ker(\lambda|_P)|\geq2.$
The cyclic group \(P\) has a unique subgroup of order \(2\). Since
\(K\leq P\) and \(|K|=2\), this subgroup is \(K\). Hence $K\leq\ker\lambda.$ Since $K\leq Z(B,+)$ and $K\leq\ker\lambda,$
we have $K\leq Z(B,+)\cap\ker\lambda=\operatorname{Soc}(B).$
Moreover, \(K\operatorname{char}(B,+)\), and hence \(K\) is invariant
under the lambda action. Therefore
Proposition~\ref{prop:subbrace-socle-ideal} implies that \(K\) is an
ideal of \(B\). Since \(K\) is non-trivial and proper, this contradicts
the simplicity of \(B\).

\medskip

\noindent\textbf{Step 4}: \textit{Conclusion of the proof.}

\medskip

Combining the previous steps, the only remaining possibility is $(B,+)\simeq \operatorname{PSL}_2(p)$
for some prime \(p\geq 5\). This completes the proof.
\end{proof}

\begin{Example} \rm \label{ex:Tsang-PSL2-Z-group}
The conclusion of Theorem~\ref{thm:B} is
actually attained for every prime \(p\geq5\).
Let $N=\operatorname{PSL}_2(p),$
regarded as the inner automorphism group of \(N\) inside
\(\operatorname{PGL}_2(p)\leq\operatorname{Aut}(N)\). As described in
\cite[Proof of Theorem~1.3]{TsangSimpleType}, the group
\(\operatorname{PGL}_2(p)\) admits an exact factorization
\[
\operatorname{PGL}_2(p)=AB,
\qquad
A\cap B=1,
\]
where \(A\) is a Singer cycle of order \(p+1\) and \(B\) is the
stabilizer of a point of the projective line. In particular, $A\simeq C_{p+1}$ and $B\simeq C_p\rtimes C_{p-1}.$ Set $A_0=A\cap N$ and $B_0=B\cap N.$
Then
\[
|A_0|=\frac{p+1}{2}
\qquad\text{and}\qquad
|B_0|=\frac{p(p-1)}{2}.
\]

Tsang's construction yields a regular subgroup of
\(\operatorname{Hol}(N)\) isomorphic to
\[
G_p=
\begin{cases}
A_0\rtimes_{\alpha} B, & p\equiv1\pmod4,\\[2mm]
B_0\rtimes_{\alpha} A, & p\equiv3\pmod4,
\end{cases}
\]
for a suitable homomorphism \(\alpha\). Consequently, there exists a
skew brace \(B_p\) such that
\[
(B_p,+)\simeq\operatorname{PSL}_2(p)
\qquad\text{and}\qquad
(B_p,\cdot)\simeq G_p.
\]

We claim that \(G_p\) is a \(Z\)-group. If \(p\equiv1\pmod4\), then $A_0\simeq C_{(p+1)/2}$, $B\simeq C_p\rtimes C_{p-1},$
and
\[
\gcd\left(\frac{p+1}{2},p(p-1)\right)=1.
\]
Both \(A_0\) and \(B\) are \(Z\)-groups, and their orders are coprime.
Hence every Sylow subgroup of \(A_0\rtimes_{\alpha}B\) is cyclic.
Similarly, if \(p\equiv3\pmod4\), then $B_0\simeq C_p\rtimes C_{(p-1)/2},$ $A\simeq C_{p+1},$
and
\[
\gcd\left(\frac{p(p-1)}{2},p+1\right)=1.
\]
Thus every Sylow subgroup of \(B_0\rtimes_{\alpha}A\) is cyclic. Therefore, for every prime \(p\geq5\), there exists a skew brace \(B_p\)
such that
\[
(B_p,+)\simeq\operatorname{PSL}_2(p)
\]
and \((B_p,\cdot)\) is a \(Z\)-group. Since
\(\operatorname{PSL}_2(p)\) is simple, the skew brace \(B_p\) is
simple.
\end{Example}

Although this construction provides an infinite family of examples, a
classification of all skew braces with additive group
\(\operatorname{PSL}_2(p)\) and multiplicative group a \(Z\)-group
appears considerably more difficult. Indeed, such a classification
would require determining, up to conjugation by
\(\operatorname{Aut}(\operatorname{PSL}_2(p))\), all regular
\(Z\)-subgroups of
\(\operatorname{Hol}(\operatorname{PSL}_2(p))\).

\section{Some splitting criteria} \label{Section Splitting criteria}

In this section, we weaken the assumption that the multiplicative group is a \(Z\)-group and study finite skew braces for which either the additive or the multiplicative group has a cyclic Sylow subgroup corresponding to the smallest prime divisor of the order. In \cite[Theorem~B]{DameleZGroups}, it was proved that every finite skew brace $B$ whose additive group is a \(Z\)-group is supersoluble. Consequently, if \(p\) is the smallest prime divisor of \(|B|\), then \(B\) admits a decomposition \(B\simeq H\rtimes P\), where \(H\) is a Hall \(p'\)-ideal of \(B\) and \(P\) is a Sylow \(p\)-subbrace. Theorem \ref{thm:C} extends this splitting result beyond skew braces of additive \(Z\)-group type. 

\begin{proof}[Proof of Theorem \ref{thm:C}] We consider the three cases separately.

\medskip 

\noindent\textbf{Case 1}: \textit{A Sylow \(p\)-subgroup of \((B,+)\) is cyclic.}

\medskip

Let $P_0\in \Syl_p(B,+)$ be cyclic. By Burnside $p$-complement theorem \cite[Theorem 2.6]{Huppert1967}, the additive group $(B,+)$ has a normal $p$-complement $H$. The subgroup \(H\) has odd order and is therefore soluble by the
Feit-Thompson theorem (see the main Theorem of \cite{FeitThompson1963}). Since \((B,+)/H\) is a cyclic \(p\)-group,
the additive group \((B,+)\) is soluble. Therefore by Hall Theorem for solvable groups \cite[Theorem 1.7]{Huppert1967} $H$ is characteristic in $(B,+)$. In particular $H$ is $\lambda$-invariant. We prove that $H$ is an ideal of $B$. It remains to prove that $H$ is normal in $(B,\cdot)$. Since \(H\) is \(\lambda\)-invariant, we may consider the generalized lambda
map associated with \(H\): $\lambda_H:(B,\cdot)\longrightarrow \operatorname{Aut}((B,+)/H)$.
The quotient \((B,+)/H\) is isomorphic to a cyclic \(p\)-group. Say $(B,+)/H\simeq C_{p^a}.$
Then
\[
|\operatorname{Aut}((B,+)/H)|=|\operatorname{Aut}(C_{p^a})|
=p^{a-1}(p-1).
\]
We claim that $H\leq \ker(\lambda_H).$
Indeed, the subgroup \(\lambda_H(H)\) has order dividing \(|H|\). On the other
hand, $\lambda_H(H)\leq \operatorname{Aut}((B,+)/H).$
Since \(p\nmid |H|\), and every prime divisor of \(p-1\) is smaller than \(p\),
while \(p\) is the smallest prime divisor of \(|B|\), we have $\gcd\bigl(|H|,|\operatorname{Aut}((B,+)/H)|\bigr)=1.$ Therefore 
by Proposition~\ref{IdealCoprimeAut}, it follows that \(H\) is an ideal of \(B\).
We complete the proof using the Schur--Zassenhaus theorem for finite skew braces
\cite[Theorem A]{DameleSchurZassenhaus2026}. See also \cite{FerraraTrombettiSchurZassenhaus2026}.
\medskip

\noindent\textbf{Case 2}: \textit{\(p\) is odd and a Sylow \(p\)-subgroup of \((B,\cdot)\) is cyclic.} 

\medskip 

By \cite[Theorem~2.1]{TrumanSylow}, the skew brace \(B\) has a Sylow \(p\)-subbrace \(P\). Thus $(P,\cdot)\in\operatorname{Syl}_p((B,+)).$ and \((P,+)\in\operatorname{Syl}_p((B,+)).\). By hypothesis, \((P,\cdot)\) is cyclic. By Theorem~\ref{Connection}, the group \((P,\cdot)\) can be identified with a regular subgroup of \(\operatorname{Hol}(P,+)\). In particular, \(\operatorname{Hol}(P,+)\) contains an element of order \(|P|\). Since \(p\) is odd, Theorem~\ref{thm:Kohl-cyclic-element} implies that \((P,+)\) is cyclic. The conclusion now follows from Case~1.

\medskip 

\noindent\textbf{Case 3}: \textit{\(p=2\), \(3\nmid |B|\), and a Sylow \(2\)-subgroup of \((B,\cdot)\) is cyclic.}

\medskip

Let \(P\) be a Sylow \(2\)-subbrace of \(B\). Then \((P,\cdot)\) is cyclic. By Theorems~\ref{Connection} and~\ref{thm:Byott-cyclic-element}, the group \((P,+)\) is cyclic, dihedral, generalized quaternion, or isomorphic to \(C_2\times C_2\). In particular, \((P,+)\) is metacyclic. Since \(3\nmid |B|\), Huppert's theorem \cite[Kapitel~IV, Satz~5.11]{Huppert1967}  implies that \((B,+)\) is \(2\)-nilpotent. Hence \((B,+)\) has a normal Hall \(2'\)-subgroup \(H\). The normal Hall \(2'\)-subgroup \(H\) has odd order and is soluble by
the Feit-Thompson theorem (see the main Theorem of \cite{FeitThompson1963}). Since \((B,+)/H\) is a \(2\)-group,
\((B,+)\) is soluble. Therefore by Hall Theorem for soluble groups \cite[Theorem 1.7]{Huppert1967} $H$ is characteristic in $(B,+)$. In particular $H$ is $\lambda$-invariant. We prove that $H$ is an ideal of $B$. Since a Sylow \(2\)-subgroup of \((B,\cdot)\) is cyclic, the Burnside normal \(2\)-complement theorem \cite[Theorem 2.6]{Huppert1967} implies that \((B,\cdot)\) has a normal Hall \(2'\)-subgroup \(N\). Since \(N\) is a normal Hall \(2'\)-subgroup of \((B,\cdot)\), it is
the unique Hall \(2'\)-subgroup. Therefore \(H=N\). Consequently $H\trianglelefteq(B,\cdot).$ Thus \(H\) is an ideal of \(B\). We complete the proof using the Schur--Zassenhaus theorem for finite skew braces
\cite[Theorem A]{DameleSchurZassenhaus2026}. See also \cite{FerraraTrombettiSchurZassenhaus2026}.

\end{proof}

\begin{Corollary}\label{cor:unified-simple-cyclic-Sylow}

Let \(B\) be a finite simple skew brace, and let \(p\) be the smallest prime divisor of \(|B|\). Suppose that one of the following conditions holds:

\begin{enumerate} 

\item a Sylow \(p\)-subgroup of \((B,+)\) is cyclic; 
\item \(p\) is odd and a Sylow \(p\)-subgroup of \((B,\cdot)\) is cyclic; 
\item \(p=2\), \(3\nmid |B|\), and a Sylow \(2\)-subgroup of \((B,\cdot)\) is cyclic. \end{enumerate} 
Then \(B\simeq C_p\).
\end{Corollary} 

\begin{proof} 
By Theorem~\ref{thm:C}, the skew brace \(B\) contains a Hall \(p'\)-ideal \(H\) and a Sylow \(p\)-subbrace \(P\) such that \(B\simeq H\rtimes P\). Since \(B\) is simple and \(H\neq B\), we have \(H=0\). Thus \(B=P\), and hence \(|B|\) is a power of \(p\). By \cite[Proposition~4.4]{CedoSmoktunowiczVendramin2019}, a finite simple skew brace of prime-power order has order \(p\). Therefore \(B\simeq C_p\).
\end{proof}

\begin{Corollary}\label{cor:order-np-not-simple} 

Let \(B\) be a finite skew brace of order \(np\), where \(\gcd(n,p)=1\), \(n>1\), and \(p\) is the smallest prime divisor of \(|B|\). Then \(B\) is not simple. 

\end{Corollary}

\begin{proof}
A Sylow \(p\)-subgroup of \((B,+)\) has order \(p\), and is therefore
cyclic. The result follows from Corollary \ref{cor:unified-simple-cyclic-Sylow}.
\end{proof}

\begin{Example} \rm \label{case 3 divides |B|}
    The hypothesis \(3\nmid |B|\) cannot be omitted. Indeed, the simple skew brace $S=S_{12,22}$ constructed  in
\cite{KonovalovSmoktunowiczVendramin} has order \(12\), additive group
isomorphic to \(A_4\), and multiplicative group isomorphic to
\(C_3\rtimes C_4\). In particular, a Sylow \(2\)-subgroup of
$(S,\cdot)$ is cyclic. Since $S$ is simple, it has no
non-trivial proper ideals and therefore cannot admit a decomposition
\(S\simeq H\rtimes P\), where \(H\) is a Hall \(2'\)-ideal and
\(P\) is a Sylow \(2\)-subbrace.
Moreover, the failure of such a splitting result is not confined to skew braces of order \(12\). Let \(T\) be any skew brace of odd order and consider the direct product $B=S\times T$. Then \((B,+)\) is soluble and a Sylow \(2\)-subgroup of \((B,\cdot)\) is still cyclic of order \(4\). Nevertheless, \(B\) has no Hall \(2'\)-ideal. Indeed, if \(H\) were such an ideal, then its image under the canonical projection \(B\to S\) would be an ideal of \(S\). Since \(S\) is simple and \(H\) has odd order, this image must be trivial. Hence \(H\leq T\), which is impossible because $|H|=|B|_{2'}=3|T|>|T|.$ Thus, even in arbitrarily large orders, no general splitting theorem of this form can hold when \(3\mid |B|\).

\end{Example}

\begin{Remark} \rm
At first sight, the splitting results of this section might suggest
that the assumption that \((B,\cdot)\) is a \(Z\)-group in
Theorems~\ref{thm:A} and~\ref{thm:B} is unnecessarily strong. Indeed, if
\(p\) is the smallest prime divisor of \(|B|\), then the existence of
a cyclic Sylow \(p\)-subgroup often forces a decomposition of \(B\) as
a semidirect product of a Hall \(p'\)-ideal and a Sylow
\(p\)-subbrace.
There is, however, an essential difference between the odd and even
cases. If \(p\) is odd and a Sylow \(p\)-subgroup of \((B,\cdot)\) is
cyclic, then the corresponding additive Sylow subgroup is cyclic as
well, and Theorem~\ref{thm:C} yields the desired splitting. If \(p=2\),
the cyclicity of a Sylow \(2\)-subgroup of \((B,\cdot)\) only implies
that a Sylow \(2\)-subgroup of \((B,+)\) is cyclic, dihedral,
generalized quaternion, or isomorphic to \(C_2\times C_2\). This is
sufficient to obtain a normal \(2'\)-complement when
\(3\nmid |B|\), but no analogous splitting theorem holds in general
when \(3\mid |B|\), as shown by Example~\ref{case 3 divides |B|}.
Moreover, in the latter situation the cyclicity of a multiplicative
Sylow \(2\)-subgroup gives no control over the Sylow subgroups of odd
order of \((B,+)\). This is precisely the additional information
provided by the assumption that \((B,\cdot)\) is a \(Z\)-group: every
odd Sylow subgroup of the additive group is then cyclic. Consequently,
extending the simplicity results of Section~\ref{Section simple multiplicative Z group} to skew braces
whose multiplicative group merely has a cyclic Sylow \(2\)-subgroup
appears to require substantially different arguments, especially when
\(3\mid |B|\).
\end{Remark}

We recall the following conjecture, suggested by an observation of
Byott in \cite[end of Section~1]{ByottSolubility2015}. Here, the type
of a Hopf--Galois structure is the isomorphism type of the associated
regular group.

\begin{Conjecture}[Byott]\label{conj:Byott-solubility}
Let \(L/K\) be a finite Galois extension. If \(L/K\) admits a
Hopf--Galois structure of soluble type, then its Galois group
\(\operatorname{Gal}(L/K)\) is soluble.
\end{Conjecture}

By the correspondence between Hopf--Galois structures on finite
Galois extensions and regular subgroups of holomorphs, this is
equivalent to the following skew-brace formulation: if \(B\) is a
finite skew brace and its additive group \((B,+)\) is soluble, then
its multiplicative group \((B,\cdot)\) is soluble. Indeed, the
additive group is the type of the corresponding Hopf--Galois
structure, whereas the multiplicative group is isomorphic to the
Galois group. Byott's conjecture remains open in full generality and continues to
be the subject of considerable research. Several partial results have
been obtained under additional assumptions on the order or on the
structure of the underlying groups; see, for instance,
\cite{GorshkovNasybullovSolvable,TsangQinSolvability,
ByottInsolubleTransitive,TrappeniersTwoSided}. Topological and
Lie-theoretic analogues of the conjecture have also recently been
investigated; see
\cite{DameleLoiTopologicalSolvability,DameleLoiLieSkewBraces}. Theorem~\ref{thm:C} gives a positive answer to the skew-brace
formulation of the conjecture when the additive group has a cyclic
Sylow \(2\)-subgroup.

\begin{Corollary}\label{cor:Byott-cyclic-additive-Sylow}
Let \(B\) be a finite skew brace such that a Sylow \(2\)-subgroup of
\((B,+)\) is cyclic. Then \((B,\cdot)\) is soluble.
\end{Corollary}

\begin{proof}
By Theorem~\ref{thm:C}, the skew brace \(B\) contains a Hall
\(2'\)-ideal \(H\) and a Sylow \(2\)-subbrace \(P\) such that $B\simeq H\rtimes P.$
In particular $(B,\cdot)\simeq (H,\cdot)\rtimes(P,\cdot).$
The group \((H,\cdot)\) has odd order and is therefore soluble by the
Feit--Thompson theorem; see the main theorem of
\cite{FeitThompson1963}. Moreover, \((P,\cdot)\) is a \(2\)-group and
is therefore soluble. Consequently, \((B,\cdot)\) is soluble.
\end{proof}

Through the correspondence with Hopf--Galois structures, we obtain
the following equivalent consequence.

\begin{Corollary}\label{cor:Byott-Hopf-Galois-cyclic-Sylow}
Let \(L/K\) be a finite Galois extension admitting a Hopf--Galois
structure of type \(N\). If a Sylow \(2\)-subgroup of \(N\) is cyclic,
then \(\operatorname{Gal}(L/K)\) is soluble.
\end{Corollary}

\begin{proof}
The given Hopf--Galois structure determines a finite skew brace \(B\)
such that $(B,+)\simeq N$ and $(B,\cdot)\simeq\operatorname{Gal}(L/K).$
Since a Sylow \(2\)-subgroup of \((B,+)\) is cyclic,
Corollary~\ref{cor:Byott-cyclic-additive-Sylow} implies that
\((B,\cdot)\), and hence \(\operatorname{Gal}(L/K)\), is soluble.
\end{proof}

\end{document}